\documentclass{article}[11pt]

\usepackage[
left=1.25in, right=1.25in]{geometry}
\usepackage[all]{xy}
\usepackage{graphicx}
\usepackage{indentfirst}
\usepackage{bm}
\usepackage{amsthm}
\usepackage{mathrsfs}
\usepackage{latexsym}
\usepackage{amsmath}
\usepackage{amssymb}
\usepackage{color}
\usepackage{hyperref}
\usepackage{microtype}

\usepackage{comment}
\usepackage{color}

\begin{document}

\newtheorem{theorem}{Theorem}[section]
\newtheorem{main}{Main Theorem}
\newtheorem{proposition}[theorem]{Proposition}
\newtheorem{corollary}[theorem]{Corollary}
\newtheorem{definition}[theorem]{Definition}
\newtheorem{lemma}[theorem]{Lemma}
\newtheorem{example}[theorem]{Example}
\newtheorem{remark}[theorem]{Remark}
\newtheorem{question}[theorem]{Question}
\newtheorem{conjecture}[theorem]{Conjecture}
\newtheorem{fact}[theorem]{Fact}
\newtheorem*{ac}{Acknowledgements}

\title{Uncertainty Principles for Kac Algebras} 
\author{
Zhengwei Liu\\
{\small Department of Mathematics and Department of Physics}\\
{\small Harvard University}\\
{\small zhengwei.liu@fas.harvard.edu}\\
\and
Jinsong Wu\\
{\small School of Mathematical Sciences}\\
{\small University of Science and Technology of China}\\
{\small wjsl@ustc.edu.cn}}
\maketitle

\begin{abstract}
In this paper, we introduce the notation of bi-shift of biprojections in subfactor theory to unimodular Kac algebras.
We characterize the minimizers of Hirschman-Beckner uncertainty principle and Donoho-Stark uncertainty principle for unimodular Kac algebras with biprojections and prove Hardy's uncertainty principle in terms of minimizers.
\end{abstract}

\section{Introduction}
Uncertainty principles for locally compact abelian groups were studied by Hardy \cite{Hardy}, Hirschman \cite{Hirsch}, Beckner \cite{Beck}, Donoho and Stark \cite{DoSt}, Smith \cite{Smith}, Tao \cite{Tao05} etc.
In 2008, Alagic and Russell \cite{AlRu} proved Donoho-Stark uncertainty principle for compact groups.
In 2004, \"{O}zaydm and Przebinda \cite{OzPr} characterized the minimizers of Hirschman-Beckner uncertainty principle and Donoho-Stark uncertainty principle for locally compact abelian groups.

Kac algebras were introduced independently by L.I Vainerman and G.I. Kac \cite{Vain74,VainKac, VainKac74} and by Enock and Nest \cite{EnSch73,EnSch74,EnSch75}, which generalized locally compact groups and their duals. Furthermore, J. Kustermans and S. Vaes introduced locally compact quantum groups \cite{KuVaes}.
Recently Crann and Kalantar proved Hirschman-Beckner uncertainty principle and Donoho-Stark uncertainty principle for unimodular locally compact quantum groups \cite{CrKa}.

Subfactor theory also provides a natural framework to study quantum symmetry.
The group symmetry is captured by the subfactor arisen from the group crossed product construction.
Ocneanu first pointed out the one-to-one correspondence between finite dimensional Kac algebras and finite-index, depth-two, irreducible subfactors. This correspondence was proved by W. Szymanski \cite{Sz94}.
Enock and Nest generalized the correspondence to infinite dimensional compact (or discrete) type Kac algebras and infinite-index, depth-two, irreducible subfactors \cite{EnNe}.
In general, a subfactor provides a pair of non-commutative spaces dual to each other and a Fourier transform $\mathcal{F}$ between them.
It appears to be natural to study Fourier analysis for subfactors.

In \cite{JLW}, C. Jiang and the authors study uncertainty principles for finite index subfactors in terms of planar algebras.
We proved Hirschman-Beckner uncertainty principle and Donoho-Stark uncertainty principle for finite index subfactors.
Furthermore, we introduced bi-shifts of biprojections \footnote{Bisch and Jones introduced \emph{biprojections} \cite{Bis94,BJ97a} which generalize the indicator function of subgroups. Bi-shifts of biprojections generalize the notion of modulation and translation of the indicator function of subgroups, although modulation and translation do not make sense in subfactor theory.}
, and use them to characterize the minimizers of the two uncertainty principles.

Moreover, we formalized Hardy's uncertainty principle using the minimizers of the Hirschman-Beckner uncertainty principle,
and proved it for finite index subfactors.
The case for finite-index, depth-two, irreducible subfactors covers the results for finite dimensional Kac algebras.
The quantum group community wondered whether the methods in \cite{JLW} work for infinite-dimensional cases.
That is the motivation of this paper.

In this paper, we introduce notions in subfactor theory to unimodular Kac algebras, such as biprojections, bi-shifts of biprojections.
For example, the identity of a compact type locally compact quantum group is a biprojection. The Fourier transform transform of a biprojection is a biprojection.
We characterize the minimizers the Hirschman-Beckner uncertainty principle and the Donoho-Stark uncertainty principle for unimodular Kac algebras containing biprojections. Furthermore, we prove the Hardy uncertainty principle for such Kac algebras.
Our proofs utilize the ideas in subfactor theory \cite{JLW} and the methods for locally compact quantum groups \cite{KuVaes}.

\begin{main}[Proposition \ref{min1}, Theorem \ref{bich}]\label{Thm:main1}
Let $\mathbb{G}$ be a unimodular Kac algebra. For any nonzero $w$ in $L^1(\mathbb{G})\cap L^2(\mathbb{G})$, the following statements are equivalent:
\begin{itemize}
\item[(1)] $H(|w|^2)+H(|\mathcal{F}(w)|^2)=-4\|w\|_2^2\log\|w\|_2$;
\item[(2)] $\mathcal{S}(w)\mathcal{S}(\mathcal{F}(w))=1$;
\item[(3)] $w$ is an extremal bi-partial isometry.
\item[(4)] $w$ is a bi-shift of a biprojection.
\end{itemize}
\end{main}

Conditions (1) and (2) are inequalities in general, namely Hirschman-Beckner uncertainty principle and Donoho-Stark uncertainty principle.
When $\mathbb{G}$ has biprojections, the above four conditions characterize the minimizers of the Hirschman-Beckner uncertainty principle.
In terms of these minimizers, we obtain Hardy's uncertainty principle for unimodular Kac algebras.
\begin{main}[Hardy's uncertainty principle, Theorem \ref{hardy}]\label{Thm:main2}
Let $\mathbb{G}$ be a unimodular Kac algebra.
Suppose that a non-zero $w$ in $L^1(\mathbb{G})\cap L^\infty(\mathbb{G})$ satisfies the conditions in Theorem \ref{Thm:main1}.
For any $x\in L^1(\mathbb{G})\cap L^\infty(\mathbb{G})$,
if $|x|\leq C|w|$ and $|\mathcal{F}(x)|\leq C'|\mathcal{F}(w)|,$ for some constants $C>0$ and $C'>0$,
then $x$ is a scalar multiple of $w$.
\end{main}

\begin{ac}
Parts of the work was done during visits of authors to Hebei Normal University. The authors would like to thank Quanhua Xu for helpful discussions. Zhengwei Liu was supported by a grant from Templeton Religion Trust. Jinsong Wu was supported by NSFC (Grant no. A010602).
\end{ac}

\section{Preliminaries}

Let $\mathcal{M}$ be a von Neumann algebra acting on a Hilbert space $\mathcal{H}$ with a normal semifinite faithful tracial weight $\varphi$.

A closed densely defined operator $x$ affiliated with $\mathcal{M}$ is called $\varphi$-measurable if for all $\epsilon>0$ there exists a projection $p\in\mathcal{M}$ such that $p\mathcal{H}\subset \mathcal{D}(x)$, and $\varphi(1-p)\leq \epsilon$, where $\mathcal{D}(x)$ is the domain of $x$.
Denote by $\widetilde{\mathcal{M}}$ the set of $\varphi$-measurable closed densely defined operators. Then $\widetilde{\mathcal{M}}$ is $*$-algebra with respect to strong sum, strong product, and adjoint operation.
If $x$ is a positive self-adjoint $\varphi$-measurable operator, then $x^\alpha\log x$ is $\varphi$-measurable for any $\alpha\in\mathbb{C}$ with $\Re\alpha>0$, where $\Re \alpha$ is the real part of $\alpha$.

The sets
$$N(\varepsilon,\varepsilon')=\{x\in\widetilde{\mathcal{M}}|\exists \text{ a projection }p\in\mathcal{M}:p\mathcal{H}\subseteq \mathcal{D}(x),\|xp\|\leq \varepsilon,\varphi(1-p)\leq \varepsilon'\},$$
where $\epsilon,\epsilon'>0$, form a basis for the neighborhoods of $0$ for a topology on $\widetilde{\mathcal{M}}$ that turns $\widetilde{\mathcal{M}}$ into a topological vector space.
Now $\widetilde{\mathcal{M}}$ is a complete Hausdorff topological *-algebra and $\mathcal{M}$ is a dense subset of $\widetilde{\mathcal{M}}$.

For any positive self-adjoint operator $x$ affiliated with $\mathcal{M}$, we put
$$\varphi(x)=\sup_{n\in\mathbb{N}}\varphi(\int_0^n tde_t),$$
where $x=\int_0^\infty tde_t$ is the spectral decomposition of $x$. Then for $p\in [1,\infty)$, the noncommutative $L^p$ space $L^p(\mathcal{M})$ with respect to $\varphi$ is given by
$$L^p(\mathcal{M})=\{x\text{ densely defined, closed, affiliated with }\mathcal{M}|\varphi(|x|^p)<\infty\}.$$
The $p$-norm $\|x\|_p$ of $x$ in $L^p(\mathcal{M})$ is given by $\|x\|_p=\varphi(|x|^p)^{1/p}$.
We have that $L^p(\mathcal{M})\subseteq \widetilde{\mathcal{M}}$.
For more details on noncommutative L$^p$ space we refer to \cite{TerpNC,Terp82}.

Throughout the paper, we will use the results in \cite{KuVaes} frequently.
Let us recall the definition of locally compact quantum groups.

Let $\mathcal{M}$ be a von Neumann algebra with a normal semifinite faithful weight $\varphi$.
Then $\mathfrak{N}_\varphi=\{x\in \mathcal{M}|\varphi(x^*x)<\infty\}$,
$\mathfrak{M}_\varphi=\mathfrak{N}_\varphi^*\mathfrak{N}_\varphi$,
$\mathfrak{M}_\varphi^+=\{x\geq 0|x\in\mathfrak{M}_\varphi\}$.
Denote by $\mathcal{H}_\varphi$ the Hilbert space by taking the closure of $\mathfrak{N}_\varphi$.
The map $\Lambda_\varphi:\mathfrak{N}_\varphi\mapsto \mathcal{H}_\varphi$ is the inclusion map.
We may use $\Lambda$ instead of $\Lambda_\varphi$ if there is no confusion.

A locally compact quantum group $\mathbb{G}=(\mathcal{M},\Delta,\varphi,\psi)$ consists of
\begin{itemize}
\item[(1)] a von Neumann algebra $\mathcal{M}$,
\item[(2)] a normal, unital, *-homomorphism $\Delta: \mathcal{M}\to\mathcal{M}\overline{\otimes}\mathcal{M}$ such that
$(\Delta\otimes\iota)\circ\Delta=(\iota\otimes\Delta)\circ\Delta$,
\item[(3)] a normal, semi-finite, faithful weight $\varphi$ such that
$(\iota\otimes\varphi)\Delta(x)=\varphi(x) 1$, $\forall x\in \mathfrak{M}_\varphi^+$;\\
a normal, semi-finite, faithful weight $\psi$ such that
$(\psi\otimes \iota)\Delta(x)=\psi(x)1$, $\forall x\in \mathfrak{M}_\psi^+$,
\end{itemize}
where $\overline{\otimes}$ denotes the von Neumann algebra tensor product, $\iota$ denotes the identity map.
The normal, unital, *-homomorphism $\Delta$ is a comultiplication of $\mathcal{M}$, $\varphi$ is the left Haar weight, and $\psi$ is the right Haar weight.

We assume that $\mathcal{M}$ acts on $\mathcal{H}_\varphi$.
There exists a unique unitary operator $W\in \mathcal{B}(\mathcal{H}_\varphi\otimes\mathcal{H}_\varphi)$ which is known as the multiplicative unitary defined by
$$W^*(\Lambda_\varphi(a)\otimes\Lambda_\varphi(b))=(\Lambda_\varphi\otimes\Lambda_\varphi)(\Delta(b)(a\otimes 1)), \quad a,b\in\mathfrak{N}_\varphi.$$
Moreover for any $x\in\mathcal{M}$, $\Delta(x)=W^*(1\otimes x)W.$

For the locally compact quantum group $\mathbb{G}$, there exist an antipode $S$, a scaling automorphism group $\tau$ and a unitary antipode $R$ and there also exists a dual locally compact quantum group $\hat{\mathbb{G}}=(\hat{\mathcal{M}},\hat{\Delta},\hat{\varphi},\hat{\psi})$ of $\mathbb{G}$.
The antipode, the scaling group, and the unitary antipode of $\hat{\mathbb{G}}$ will denoted by $\hat{S}$, $\hat{\tau}$, and $\hat{R}$ respectively. We refer \cite{KuVaes,KuVaes03} for more details.

For any $\omega\in\mathcal{M}_*$, $\lambda(\omega)=(\omega\otimes\iota)(W)$ is the Fourier representation of $\omega$, where $\mathcal{M}_*$ is the Banach space of all bounded normal functional on $\mathcal{M}$.
For any $\omega$, $\theta$ in $\mathcal{M}_*$, the convolution $\omega*\theta$ is given by
$$\omega*\theta=(\omega\otimes\theta)\Delta.$$
In \cite{LWW}, S. Wang and the authors defined the convolution $x*y$ of $x\in L^p(\mathbb{G})$ and $L^q(\mathbb{G})$ for $1\leq p,q\leq 2$. If the left Haar weights $\varphi$, $\hat{\varphi}$ of $\mathbb{G}$ and $\hat{\mathbb{G}}$ respectively are tracial weights, we have that the convolution is well-defined for $1\leq p,q\leq \infty$ by the results in \cite{LWW}.

For any locally compact quantum group $\mathbb{G}$, the Fourier transforms $\mathcal{F}_p: L^p(\mathbb{G})\to L^q(\hat{\mathbb{G}})$ is well-defined. (See \cite{Cooney},\cite{Cas} for the definition of Fourier transforms and \cite{Daele07} for the definition of the Fourier transform for algebraic quantum groups.)
For any $x$ in $L^1(\mathbb{G})$, we deonte by $x\varphi$ the bounded linear functional on $L^\infty(\mathbb{G})$ given by $(x\varphi) (y)=\varphi(yx)$ for any $y$ in $L^\infty(\mathbb{G})$.
Recall that a projection $p$ in $L^1(\mathbb{G})\cap L^\infty(\mathbb{G})$ is a biprojection if $\mathcal{F}_1(p\varphi)$ is a multiple of a projection in $L^\infty(\hat{\mathbb{G}})$, (see \cite{LWW} for more properties of biprojections).

\section{Main Results}
In this section, we will focus on a unimodular Kac algebra $\mathbb{G}$, which is a locally compact quantum group subject to the condition $\varphi=\psi$ is tracial. (See \cite{EnSchwart} for more details.)
We denote $L^\infty(\mathbb{G})$ by $\mathcal{M}$.
The Fourier transform $\mathcal{F}_p$ from $L^p(\mathbb{G})$ to $L^q(\hat{\mathbb{G}})$ is given by $x\mapsto \lambda(x\varphi)$ for any $x\in L^1(\mathbb{G})\cap L^\infty(\mathbb{G})$.
For a unimodular Kac algebra $\mathbb{G}$, we will denote by $\mathcal{F}$ the Fourier transform for simplicity.

For any $\varphi$-measurable element $x$ in $\widetilde{\mathcal{M}}$, the von Neumann entropy $H(|x|^2)$ is defined by
$$H(|x|^2)=-\varphi(x^*x\log x^*x).$$

\begin{proposition}\label{Hirsch}
Let $\mathbb{G}$ be a unimodular Kac algebra. Then for any $x\in L^1(\mathbb{G})\cap L^2(\mathbb{G})$, we have
$$H(|x|^2)+H(|\mathcal{F}(x)|^2)\geq -4\|x\|_2^2\log \|x\|_2.$$
\end{proposition}
\begin{proof}
By Lemma 18 in \cite{TerpNC}, we have that $\alpha\mapsto |x|^\alpha$ is differentiable for $\alpha>0$.
Now differentiating the Hausdorff-Young inequality \cite{Cooney}
$$\|\mathcal{F}(x)\|_q\leq \|x\|_p,\quad x\in L^1(\mathbb{G})\cap L^2(\mathbb{G}), \quad p\in[1,2],\quad\frac{1}{p}+\frac{1}{q}=1,$$
with respect to $p$ and plug $p=2$ into the result inequality, we can obtain that
$$H(|x|^2)+H(|\mathcal{F}(x)|^2)\geq -4\|x\|_2^2\log \|x\|_2.$$
\end{proof}

For any $x\in\widetilde{\mathcal{M}}$, let $\mathcal{S}(x)=\varphi(\mathcal{R}(x))$, where $\mathcal{R}(x)$ is the range projection of $x$.

\begin{proposition}\label{DS}
Let $\mathbb{G}$ be a unimodular Kac algebra.
Then for any nonzero $x\in L^1(\mathbb{G})\cap L^2(\mathbb{G})$, we have
$$\mathcal{S}(x)\mathcal{S}(\mathcal{F}(x))\geq 1.$$
\end{proposition}
\begin{proof}
We present two proofs here.

1.  By using the inequality $\log \mathcal{S}(x)\geq H(|x|^2)$ when $\|x\|_2=1$ and Proposition \ref{Hirsch},
we see the proposition is true.

2. We assume that $\mathcal{S}(x),\mathcal{S}(\mathcal{F}(x))<\infty$. Then by H\"{o}lder's inequality, we have
\begin{eqnarray*}
\|\mathcal{F}(x)\|_\infty&\leq &\|x\|_1\leq \|\mathcal{R}(x)\|_2\|x\|_2\\
&=&\mathcal{S}(x)^{1/2}\|\mathcal{F}(x)\|_2\\
&\leq&\mathcal{S}(x)^{1/2}\mathcal{S}(\mathcal{F}(x))^{1/2}\|\mathcal{F}(x)\|_\infty.
\end{eqnarray*}
Therefore $\mathcal{S}(x)\mathcal{S}(\mathcal{F}(x))\geq 1$.
\end{proof}

\begin{definition}
An element $x$ in $L^1(\mathbb{G})\cap L^2(\mathbb{G})$ is said to be extremal if $\|\mathcal{F}(x)\|_\infty=\|x\|_1$. We say a nonzero element $x$ is an (extremal) bi-partial isometry if $x$ and $\mathcal{F}(x)$ are multiplies of (extremal) partial isometries.
\end{definition}

\begin{proposition}
Let $\mathbb{G}$ be a unimodular Kac algebra. If $x$ is extremal, then $x^*$ and $R(x)$ are extremal.
\end{proposition}
\begin{proof}
By Proposition 2.4 in \cite{KuVaes03}, we have
\begin{eqnarray*}
\|\mathcal{F}(x^*)\|_\infty&=&\|\lambda(x^*\varphi)\|_\infty=\|\lambda(x^*\varphi)^*\|_\infty\\
&=&\|\lambda(\overline{x^*\varphi}R)\|_\infty=\|\lambda(x\varphi R)\|_\infty\\
&=&\|\hat{R}(\lambda(x\varphi))\|_\infty=\|\lambda(x\varphi)\|_\infty,\\
\|\mathcal{F}(R(x))\|_\infty&=&\|\lambda(R(x)\varphi)\|_\infty=\|\lambda(x\varphi R)\|_\infty\\
&=&\|\hat{R}(\lambda(x\varphi))\|_\infty=\|\lambda(x\varphi)\|_\infty,
\end{eqnarray*}
and
$$\varphi(|x|)=\varphi(|x^*|)=\varphi(R(|x|))=\varphi(|R(x)|)$$
Therefore $x^*$ and $R(x)$ are extremal.
\end{proof}

\begin{proposition}\label{min1}
Let $\mathbb{G}$ be a unimodular Kac algebra. For any nonzero $x$ in $L^1(\mathbb{G})\cap L^2(\mathbb{G})$, the following statements are equivalent:
\begin{itemize}
\item[(1)] $H(|x|^2)+H(|\mathcal{F}(x)|^2)=-4\|x\|_2^2\log\|x\|_2$;
\item[(2)] $\mathcal{S}(x)\mathcal{S}(\mathcal{F}(x))=1$;
\item[(3)] $x$ is an extremal bi-partial isometry.
\end{itemize}
\end{proposition}
\begin{proof}
"(1)$\Rightarrow $(3)". We assume that $\|x\|_2=1$. Now we follow the proof in \cite{JLW}. First, we define a complex function $F(z)$ for $z=\sigma+it$, $\frac{1}{2}<\sigma<1$ as
$$F(z)=\hat{\varphi}(\mathcal{F}(w_x|x|^{2z})|\mathcal{F}(x)|^{2z}w_{\mathcal{F}(x)}^*),$$
where $w_x$ means the partial isometry in the polar decomposition of $x$. Note that $x\in L^1(\mathbb{G})\cap L^2(\mathbb{G})$, we see that $\mathcal{F}(w_x|x|^{2z})$ is well-defined.

By H\"{o}lder's inequality and the Hausdorff-Young inequality \cite{Cooney}, we have
\begin{eqnarray*}
|F(\sigma+it)|\leq \|\mathcal{F}(w_x|x|^{2z})\|_{\frac{1}{1-\sigma}}\||\mathcal{F}(x)|^{2z}\|_{\frac{1}{\sigma}}\leq \||x|^{2\sigma}\|_{\frac{1}{\sigma}}\||\mathcal{F}(x)|^{2\sigma}\|_{\frac{1}{\sigma}}=1.
\end{eqnarray*}
This implies $F(z)$ is bounded on $\frac{1}{2}<\sigma<1$. By Lemma 18 in \cite{TerpNC} again, we can follow the proof of Theorem 6.4 in \cite{JLW} directly to obtain that
$$\hat{\varphi}(\mathcal{F}(x|x|)|\mathcal{F}(x)|\mathcal{F}(x)^*)=1.$$
Now we see that
\begin{equation}\label{keyh}
\begin{aligned}
1&=\hat{\varphi}(\mathcal{F}(x|x|)|\mathcal{F}(x)|\mathcal{F}(x)^*)\\
&=(x|x|\varphi\otimes (|\mathcal{F}(x)|\mathcal{F}(x)^*)\hat{\varphi})(W)\\
&=(w_x|x|^2\varphi\otimes (|\mathcal{F}(x)|^2w_{\mathcal{F}(x)}^*)\hat{\varphi})(W)\\
&=(|x|^2\varphi\otimes (|\mathcal{F}(x)|^2)\hat{\varphi})((1\otimes w_{\mathcal{F}(x)}^*)W(w_x\otimes 1))\\
&\leq(|x|^2\varphi\otimes (|\mathcal{F}(x)|^2)\hat{\varphi})
(1\otimes 1)=1.
\end{aligned}
\end{equation}

Let $p=w_x^*w_x$ and $q=w_{\mathcal{F}(x)}^*w_{\mathcal{F}(x)}$.
Since the equality holds in Inequality (\ref{keyh}),
we have that
$$(p\otimes w_{\mathcal{F}(x)}^*)W(w_x\otimes q)=p\otimes q.$$
Applying $|x|\varphi\otimes\iota$ to the both sides of the equation above, we obtain that
$$w_{\mathcal{F}(x)}^*\mathcal{F}(x)q=\varphi(|x|)q,$$
i.e. $\mathcal{F}(x)=\varphi(|x|)w_{\mathcal{F}(x)}.$
Similarly, we can obtain that $x=\hat{\varphi}(|\mathcal{F}(x)|)w_x$.
Now we see that $x$ is an extremal bi-partial isometry.

"(3)$\Rightarrow$(2)". Suppose $x$ is an extremal bi-partial isometry. Following the second proof in Proposition \ref{DS}, we have
\begin{eqnarray*}
\|\mathcal{F}(x)\|_\infty&=&\|x\|_1=\|\mathcal{R}(x)\|_2\|x\|_2\\
&=&\varphi(\mathcal{R}(x))^{1/2}\|\mathcal{F}(x)\|_2\\
&=&\varphi(\mathcal{R}(x))^{1/2}\hat{\varphi}(\mathcal{R}(\mathcal{F}(x)))^{1/2}\|\mathcal{F}(x)\|_\infty.
\end{eqnarray*}
Hence $\mathcal{S}(x)\mathcal{S}(\mathcal{F}(x))=1$.

"(2)$\Rightarrow$(1)". Since (2) is weaker than (1), we see that (2) implies (1).
\end{proof}

\begin{definition}
Let $\mathbb{G}$ be a unimodular Kac algebra with a biprojection
$B$ in $L^1(\mathbb{G})\cap L^\infty(\mathbb{G})$.
A projection $x$ in $L^1(\mathbb{G})\cap L^2(\mathbb{G})$ is called a left shift of a biprojection $B$ if $\varphi(x)=\varphi(B)$ and $x*B=\varphi(B)x$.
A projection $x$ in $L^1(\mathbb{G})\cap L^2(\mathbb{G})$ is called a right shift of a biprojection $B$ if $\varphi(x)=\varphi(B)$ and $B*x=\varphi(B)x$.
\end{definition}

\begin{proposition}
Let $\mathbb{G}$ be a unimodular Kac algebra.
Suppose that there exists a biprojection $B$ in $L^1(\mathbb{G})\cap L^\infty(\mathbb{G})$ and $x$ is a right (or left) shift of a biprojection $B$ in $L^1(\mathbb{G})\cap L^2(\mathbb{G})$. Then $x$ is an extremal bi-partial isometry.
\end{proposition}
\begin{proof}
By Proposition \ref{min1}, it suffices to show that $x$ is a minimizer of the uncertainty principle.

Since $B*x=\varphi(B)x$, we have $\mathcal{F}(B)\mathcal{F}(x)=\varphi(B)\mathcal{F}(x)$ i.e. $\mathcal{R}(\mathcal{F}(x))\leq \mathcal{R}(\mathcal{F}(B))$.

By Proposition \ref{DS},
we have $\varphi(x)\hat{\varphi}(\mathcal{R}(\mathcal{F}(x)))\geq 1$ and
$$1=\varphi(B)\hat{\varphi}(\mathcal{R}(\mathcal{F}(B)))\geq \varphi(x)\hat{\varphi}(\mathcal{R}(\mathcal{F}(x)))\geq 1.$$
Now we have $\mathcal{R}(\mathcal{F}(x))= \mathcal{R}(\mathcal{F}(B))$.
Hence $x$ is a minimizer of the uncertainty principle.
\end{proof}

\begin{definition}\label{bi-shift}
Let $\mathbb{G}$ be a unimodular Kac algebra.
Suppose there exists a biprojection $B$ in $L^1(\mathbb{G})\cap L^2(\mathbb{G})$, we denote by $\widetilde{B}$ the range projection of $\mathcal{F}(B)$.
A nonzero element $x$ in $L^\infty(\mathbb{G})$ is said to be a bi-shift of a biprojection $B$ if there exist a right shift $B_g$ of the biprojection $B$ and
a right shift $\widetilde{B}_h$ of the biprojection $\widetilde{B}$ and
an element $y$ in $L^\infty(\mathbb{G})$ such that
$$x=\widehat{\mathcal{F}}(\widetilde{B}_h)*(B_gy).$$
\end{definition}

Now we will prove that the bi-shift of a biprojection described as above is a minimizer of the uncertainty principle.
To see this, we need the following lemma.

\begin{lemma}\label{ran}
Let $\mathbb{G}$ be a unimodular Kac algebra. Suppose $x$, $y$ and $\mathcal{R}(x),\mathcal{R}(y)$ are in $L^1(\mathbb{G})\cap L^\infty(\mathbb{G})$.
Then
$$(x*y)(x*y)^*\leq \|\mathcal{R}(x^*)\|_2^2(xx^*)*(yy^*),$$
and
$$\mathcal{R}(x*y)\leq \mathcal{R}(\mathcal{R}(x)*\mathcal{R}(y)).$$
\end{lemma}
\begin{proof}
First, we assume that $x$ and $y$ are positive.
Then $x\leq \|x\|\mathcal{R}(x)$ and $y\leq \|y\|\mathcal{R}(y)$.
Now by computing the convolution \cite{LWW}, we obtain that
\begin{eqnarray*}
x*y&=&((x\varphi) R\otimes\iota)(\Delta(y))\\
&=&((x^{1/2}\varphi x^{1/2}) R\otimes\iota)(\Delta(y))\\
&\leq&\|y\|((x^{1/2}\varphi x^{1/2}) R\otimes\iota)(\Delta(\mathcal{R}(y)))\\
&=&\|y\|x*\mathcal{R}(y)\\
&=&\|y\|(\iota\otimes \mathcal{R}(y)\varphi R)(\Delta(x))\\
&\leq &\|x\|\|y\|\mathcal{R}(x)*\mathcal{R}(y).
\end{eqnarray*}
Therefore,
$$\mathcal{R}(x*y)\leq \mathcal{R}(\mathcal{R}(x)*\mathcal{R}(y)).$$

When $x$, $y$ are in the general case, we will show that
\begin{equation}\label{abs}
(x*y)(x*y)^*\leq \|\mathcal{R}(x^*)\|_2^2 (xx^*)*(yy^*).
\end{equation}
If this inequality (\ref{abs}) is true, then we can see that the second inequality in the Lemma is proved.
By Lemma 9.5 in \cite{KuVaes} and $L^1(\mathbb{G})\cap L^\infty(\mathbb{G})\subset \mathfrak{N}_\varphi$, we have
\begin{eqnarray*}
\lefteqn{R((xx^*)*(yy^*))}\\
&=&R((xx^*\varphi)R\otimes\iota)(\Delta(yy^*))\\
&=&(\iota\otimes\omega_{\Lambda(x),\Lambda(x)})(\Delta(R(y)^*R(y))\\
&\geq&\frac{1}{\|\mathcal{R}(x^*)\|_2^2}((\iota\otimes\omega_{\Lambda(x),\Lambda(\mathcal{R}(x^*))})\Delta(R(y)))^*(\iota\otimes\omega_{\Lambda(x),\Lambda(\mathcal{R}(x^*))})\Delta(R(y))\\
&=&\frac{1}{\|\mathcal{R}(x^*)\|_2^2}(R(x*y))^*R(x*y)\\
&=&\frac{1}{\|\mathcal{R}(x^*)\|_2^2}R((x*y)(x*y)^*),
\end{eqnarray*}
i.e.
$$(x*y)(x*y)^*\leq \|\mathcal{R}(x^*)\|_2^2(xx^*)*(yy^*).$$
\end{proof}

\begin{proposition}
Let $\mathbb{G}$ be a unimodular Kac algebra. Suppose $x$ is the bi-shift of the biprojection $B$ as in the Definition \ref{bi-shift}. Then $\mathcal{R}(x^*)=B_g$ and $\mathcal{R}(\mathcal{F}(x))=\widetilde{B}_h$. Moreover, $x$ is a minimizer of the uncertainty principles.
\end{proposition}
\begin{proof}
Note that $x=\widehat{\mathcal{F}}(\widetilde{B}_h)*(B_gy)$, we then have
$\mathcal{F}(x)=\widetilde{B}_h\mathcal{F}(B_gy)$.
This implies that $\mathcal{R}(\mathcal{F}(x))\leq \widetilde{B}_h$.
From the fact that $\widetilde{B}_h$ is a right shift of the biprojection $\widetilde{B}$, we see $\hat{\varphi}(\widetilde{B}_h)=\hat{\varphi}(\widetilde{B}).$

On the other hand, we have $\mathcal{R}(\widehat{\mathcal{F}}(\widetilde{B}_h))=\mathcal{R}(\widehat{\mathcal{F}}(\widetilde{B}))=\mathcal{R}(B)=B$ and by Lemma \ref{ran}
\begin{eqnarray*}
\mathcal{R}(x)&\leq& \mathcal{R}(\mathcal{R}(\widehat{\mathcal{F}}(\widetilde{B}_h))*\mathcal{R}(B_gy)))\\
&\leq &\mathcal{R}(B*B_g)=B_g.
\end{eqnarray*}
Now by Proposition \ref{DS}, we see that
\begin{eqnarray*}
1&\leq& \varphi(\mathcal{R}(x))\hat{\varphi}(\mathcal{R}(\mathcal{F}(x)))\leq  \varphi(B_g)\hat{\varphi}(\widetilde{B})\\
&=&\varphi(B)\hat{\varphi}(\widetilde{B})=1.
\end{eqnarray*}
Therefore all inequalities above must be equalities and $\mathcal{R}(x)=B_g$ and $\mathcal{R}(\mathcal{F}(x))=\widetilde{B}_h$. Moreover, $x$ is a minimizer of the uncertainty principles.
\end{proof}

\begin{proposition}\label{exbi}
Let $\mathbb{G}$ be a unimodular Kac algebra.
Suppose $w$ is a partial isometry in $L^1(\mathbb{G})\cap L^\infty(\mathbb{G})$ and $\mathcal{F}(w)$ is extremal.
Then $w$ is an extremal bi-partial isometry.
\end{proposition}
\begin{proof}
By H\"{o}lder's inequality, we have $x$ is a multiple of a partial isometry if and only if $\|x\|_2^2=\|x\|_\infty\|x\|_1$.
To see that $\mathcal{F}(w)$ is a multiple of a partial isometry, it is enough to check that
$$\|\mathcal{F}(w)\|_2^2=\|\mathcal{F}(w)\|_\infty\|\mathcal{F}(w)\|_1.$$

Since $\mathcal{F}(w)$ is extremal, we have
$$\|w\|_\infty=\|\widehat{\mathcal{F}}(\mathcal{F}(w))\|_\infty=\|\mathcal{F}(w)\|_1.$$
Now by H\"{o}lder's inequality and Hausdorff-Young inequality \cite{Cooney}, we obtain
\begin{eqnarray*}
\|\mathcal{F}(w)\|_\infty\|\mathcal{F}(w)\|_1&\geq &\|\mathcal{F}(w)\|_2^2=\|w\|_2^2\\
&=&\|w\|_\infty\|w\|_1\\
&\geq &\|\mathcal{F}(w)\|_1\|\mathcal{F}(w)\|_\infty.
\end{eqnarray*}
Hence $\|\mathcal{F}(w)\|_2^2=\|\mathcal{F}(w)\|_\infty\|\mathcal{F}(w)\|_1$ and $\|\mathcal{F}(w)\|_\infty=\|w\|_1$.
Now we see that $w$ is an extremal bi-partial isometry.
\end{proof}

\begin{theorem}\label{sq}
Let $\mathbb{G}$ be a unimodular Kac algebra. Suppose there is an extremal bi-partial isometry $w$ in $L^1(\mathbb{G})\cap L^2(\mathbb{G})$. Then
$$(w*R(w)^*)(w^**R(w))=\|w\|_2^2(ww^*)*(R(w)^*R(w)).$$
Moreover $\frac{1}{\|w\|_2^2}w*R(w)^*$ is a partial isometry and $\|w\|_1=\frac{1}{\|w\|_2}\|w*R(w)^*\|_1$.
\end{theorem}
\begin{proof}
By Lemma 9.5 in \cite{KuVaes}, we have
\begin{eqnarray*}
\lefteqn{R((ww^*)*(R(w)^*R(w)))}\\
&=&R((ww^*\varphi R\otimes \iota)(\Delta(R(w)^*R(w))))\\
&=&(\iota\otimes (\omega_{\Lambda(w),\Lambda(w)})(\Delta(ww^*)))\\
&\geq &\frac{1}{\|w\|_2^2} ((\iota\otimes\omega_{\Lambda(w),\Lambda(|w|)})\Delta(w^*))^*((\iota\otimes\omega_{\Lambda(w),\Lambda(|w|)})\Delta(w^*))\\
&=&\frac{1}{\|w\|_2^2} (R(w\varphi R\otimes\iota)(\Delta(R(w^*))))^*(R(w\varphi R\otimes\iota)(\Delta(R(w^*))))\\
&=& \frac{1}{\|w\|_2^2} R(w*R(w^*))^*R(w*R(w^*))\\
&=&\frac{1}{\|w\|_2^2} R((w*R(w^*))(w^**R(w)))
\end{eqnarray*}
i.e
\begin{equation}\label{insq1}
(w*R(w)^*)(w^**R(w))\leq \|w\|_2^2(ww^*)*(R(w)^*R(w)).
\end{equation}
We will show that the traces of the both sides are equal.
For the right hand side, we have
\begin{equation}\label{insq2}
\begin{aligned}
\varphi((ww^*)*(R(w)^*R(w)))&=\varphi(ww^*)\varphi(R(w)^*R(w))\\
&=\|w\|_2^2\|R(w)\|_2^2=\|w\|^4_2
\end{aligned}
\end{equation}
On the other hand, since $w$ is an extremal bi-partial isometry, we
let $w=\widehat{\mathcal{F}}(x)$ for $x$ in $L^1(\widehat{\mathbb{G}})$.
Then we have that
$$\mathcal{F}(w*R(w)^*)=\mathcal{F}(w)\mathcal{F}(R(w)^*)=xx^*.$$
Therefore $w*R(w)^*=\widehat{\mathcal{F}}(xx^*)$ and
\begin{eqnarray*}
\varphi((w*R(w)^*)(w^**R(w)))&=&\varphi(\widehat{\mathcal{F}}(xx^*)\widehat{\mathcal{F}}(xx^*)^*)\\
&=&\hat{\varphi}(xx^*xx^*).
\end{eqnarray*}
Note that $x$ is a multiple of a partial isometry.
We assume that $x=\mu x_0$ for some $\mu\in\mathbb{C}$ and a partial isometry $x_0$.
Then $(xx^*)^2=|\mu|^4|x_0|$.
Since $w$ is a minimizer of the uncertainty principle, we have $\varphi(|w|)\hat{\varphi}(|x_0|)=1$ i.e. $\hat{\varphi}(|x_0|)=\frac{1}{\|w\|_2^2}.$
Meanwhile we have $\|w\|_2=\|x\|_2$.
Now we can obtain that $\|w\|_2^2=|\mu|^2\frac{1}{\|w\|_2^2}$ and $|\mu|=\|w\|_2^2$.

Hence $\hat{\varphi}((xx^*)^2)=|\mu|^4\frac{1}{\|w\|_2^2}=\|w\|_2^6$ i.e. the trace of the left hand side of inequality (\ref{insq1}) is $\|w\|_2^6$. By Equation (\ref{insq2}),  we have the trace of the right hand side of inequality (\ref{insq1}) is $\|w\|_2^6$.
This implies that
$$(w*R(w)^*)(w^**R(w))=\|w\|_2^2(ww^*)*(R(w)^*R(w)).$$

Now we show that $w*R(w)^*$ is a multiple of a partial isometry.
By H\"{o}lder's inequality we have
$$\|w\|_2^6=\|w*R(w)^*\|_2^2\leq \|w*R(w)^*\|_\infty\|w*R(w)^*\|_1.$$
By Hausdorff-Young inequality \cite{Cooney}, we obtain
$$\|w*R(w)^*\|_\infty=\|\widehat{\mathcal{F}}(xx^*)\|_\infty\leq \|xx^*\|_1=\|x\|_2^2=\|w\|_2^2$$
and by Young's inequality, we have
$$\|w*R(w)^*\|_1\leq \|w\|_1\|R(w)^*\|_1=\|w\|_1^2=\|w\|_2^4.$$
Hence all equalities of the inequalities above hold and
$$\|w*R(w)^*\|_2^2=\|w*R(w)^*\|_\infty\|w*R(w)^*\|_1.$$
Finally we see that $\frac{1}{\|w\|_2^2}w*R(w)^*$ is a partial isometry and
$$\|w\|_1=\|w\|_2^2=\|\frac{1}{\|w\|_2^2}w*R(w)^*\|_1.$$
\end{proof}

Corollary 6.12 in \cite{JLW} is a useful tool to find an extremal bi-partial isometry in a given element.
However, that result is not true in general. Instead, we have the following result for unimodular Kac algebras:
\begin{corollary}\label{cor1}
Let $\mathbb{G}$ be a unimodular Kac algebra.
Suppose $w\in L^1(\mathbb{G})\cap L^2(\mathbb{G})$ such that $\|w*R(w^*)\|_\infty=\|w\|_2^2$, $\|w\|_2^2$ is a point spectrum of $w*R(w^*)$, and $Q$ is the spectral projection of $|w*R(w^*)|$ with spectrum $\|w\|_2^2$.
Then $Q$ is an extremal bi-partial isometry.
\end{corollary}
\begin{proof}
We assume that $\|w\|_2=1$.
Note that
$$\lim_{k\to\infty}((w^**R(w))(w*R(w^*)))^k=Q,$$
in the strong operator topology and $Q$ is a projection.
By the assumption that $w\in L^1(\mathbb{G})\cap L^2(\mathbb{G})$ and Young's inequality, we have that $((w^**R(w))(w*R(w^*)))^k\in L^1(\mathbb{G})$ for $k=1,2,\ldots$.
Hence $\lim_{k\to\infty}\|((w^**R(w))(w*R(w^*)))^k-Q\|_1=0$.
By the Hausdorff-Young inequality \cite{Cooney}, we obtain that
$$\lim_{k\to\infty}\|\mathcal{F}(((w^**R(w))(w*R(w^*)))^k)-\mathcal{F}(Q)\|_\infty=0,$$
i.e.
$$\mathcal{F}(Q)=\lim_{k\to\infty}((\mathcal{F}(w^*)\mathcal{F}(w^*)^*)*(\mathcal{F}(w)\mathcal{F}(w)^*))^{*(k)}>0$$
in the norm topology.

Note that
$\|((\mathcal{F}(w^*)\mathcal{F}(w^*)^*)*(\mathcal{F}(w)\mathcal{F}(w)^*))^{*(k)}\|_1=\|w\|_2^{4k}=1$.
We then see that $\|\mathcal{F}(Q)\|_1=1=\|Q\|_\infty.$
By Proposition \ref{exbi}, we see that $Q$ is an extremal bi-partial isometry.
\end{proof}

\begin{theorem}\label{bich}
Let $\mathbb{G}$ be a unimodular Kac algebra and $w\in L^1(\mathbb{G})\cap L^\infty(\mathbb{G})$.
Then $w$ is an extremal bi-partial isometry if and only if $w$ is a bi-shift of a biprojection.
Furthermore, if $w$ is a projection, then it is a left (or right) shift of a biprojection.
\end{theorem}
\begin{proof}
Suppose $w$ is an extremal bi-partial isometry and $w$ is a partial isometry. Let
$$B=\frac{1}{\|w\|_2^4}(w*R(w)^*)(w^**R(w)).$$
By Theorem \ref{sq}, we have that $\frac{1}{\|w\|_2^2}w*R(w)^*$ is a partial isometry and hence $B$ is a projection.

Now we compute the Fourier transform of $B$.
\begin{eqnarray*}
\mathcal{F}(B)&=&\frac{1}{\|w\|_2^4}\mathcal{F}((w*R(w)^*)(w^**R(w)))\\
&=&\frac{1}{\|w\|_2^2}\mathcal{F}((ww^*)*(R(w)^*R(w)))\\
&=&\frac{1}{\|w\|_2^2}\mathcal{F}(ww^*)\mathcal{F}(R(w)^*R(w))\\
&=&\frac{1}{\|w\|_2^2}\mathcal{F}(ww^*)\mathcal{F}(ww^*)^*
\end{eqnarray*}
Hence it is suffices to check $\mathcal{F}(ww^*)$ is a multiple of partial isometry.
First we observe that $\mathcal{F}(w)$ is an extremal bi-partial isometry.
By Theorem \ref{sq}, we have that $\mathcal{F}(w)*\hat{R}(\mathcal{F}(w)^*)$ is a multiple of partial isometry and
$$
\mathcal{F}(w)*\hat{R}(\mathcal{F}(w)^*)=\mathcal{F}(w)*\mathcal{F}(w^*)=\mathcal{F}(ww^*).
$$
Therefore $\mathcal{F}(B)$ is a multiple of a projection and $B$ is a biprojection.

Now we define $B_g=ww^*$, then $B_g$ is a projection. We are going to show that $B_g$ is a  right shift of the biprojection $B$. By proposition \ref{sq}, we have that $\frac{1}{\|w\|_2^2}B_g*R(B_g)=B$. Computing the trace on both sides, we have $\frac{1}{\|w\|_2^2}\varphi(B_g)^2=\varphi(B)$. Note that $\varphi(B_g)=\|w\|_2^2$, we see
$$\varphi(B)=\frac{1}{\|w\|_2^2}(\|w\|_2^2)^2=\|w\|_2^2=\varphi(B_g).$$

Recall that $\mathcal{F}(w)$ is an extremal bi-partial isometry. We have $\|\mathcal{F}(w)\|_\infty=\|w\|_1$, and $\frac{1}{\|w\|_2^2}\mathcal{F}(w)$ is a partial isometry.
By Theorem \ref{sq}, we see that
$$\frac{1}{\|\frac{1}{\|w\|_2^2}\mathcal{F}(w)\|_2^2}\frac{\mathcal{F}(w)}{\|w\|_2^2}*\frac{\hat{R}(\mathcal{F}(w)^*)}{\|w\|_2^2}=\frac{1}{\|w\|_2^2}\mathcal{F}(ww^*)=\frac{1}{\|w\|_2^2}\mathcal{F}(B_g)$$
is a partial isometry.

Hence we obtain that
\begin{eqnarray*}
\mathcal{F}(B_g)&=&\frac{1}{\|w\|_2^4}\mathcal{F}(B_g)\mathcal{F}(B_g)^*\mathcal{F}(B_g)\\
&=&\frac{1}{\|w\|_2^4}\mathcal{F}(B_g)\mathcal{F}(R(B_g))\mathcal{F}(B_g)\\
&=&\frac{1}{\|w\|_2^4}\mathcal{F}(B_g*R(B_g)*B_g)
\end{eqnarray*}
and $\frac{1}{\|w\|_2^4}B_g*R(B_g)*B_g=B_g$. Then
$$B*B_g=\frac{1}{\|w\|_2^2}B_g*R(B_g)*B_g=\|w\|_2^2B_g=\varphi(B_g)B_g.$$
Therefore $B_g$ is a right shift of the biprojection $B$.

Let $\widetilde{B}_h=\frac{1}{\|w\|_2^4}\mathcal{F}(w)\mathcal{F}(w)^*$. We have $\widehat{\mathcal{F}}(\widetilde{B}_h)=\frac{1}{\|w\|_2^4}w*R(w)^*$. Finally we will find a form of $w$ in terms of $B_g$ and $\widetilde{B}_h$.
\begin{eqnarray*}
\mathcal{F}(w)&=&\frac{1}{\|w\|_2^4}\mathcal{F}(w)\mathcal{F}(w)^*\mathcal{F}(w)\\
&=&\frac{1}{\|w\|_2^4}\mathcal{F}(w)\mathcal{F}(R(w)^*)\mathcal{F}(w)\\
&=&\frac{1}{\|w\|_2^4}\mathcal{F}(w*R(w)^**w).
\end{eqnarray*}
Then $w=\frac{1}{\|w\|_2^4}w*R(w)^**w=\widehat{\mathcal{F}}(\widetilde{B}_h)*(B_gw).$
\end{proof}

\begin{corollary}\label{cor2}
Let $\mathbb{G}$ be a unimodular Kac algebra.
If $x\in L^1(\mathbb{G})\cap L^2(\mathbb{G})$ and $\mathcal{F}(x)$ are positive and $\mathcal{S}(x)\mathcal{S}(\mathcal{F}(x))=1$, then $x$ is a biprojection.
\end{corollary}

\begin{lemma}\label{scale}
Let $\mathbb{G}$ be a unimodular Kac algebra. Suppose $B$ is a biprojection in $L^1(\mathbb{G})\cap L^\infty(\mathbb{G})$ and $\widetilde{B}$ is the range projection of $\mathcal{F}(B)$ in $L^1(\widehat{\mathbb{G}})\cap L^\infty(\widehat{\mathbb{G}})$. If $x\in L^1(\mathbb{G})\cap L^\infty(\mathbb{G})$ such that $\mathcal{R}(x)=B$ and $\mathcal{R}(\mathcal{F}(x))=\widetilde{B}$, then $x$ is a multiple of $B$.
\end{lemma}
\begin{proof}
By the assumption, we have $Bx=x$ and $\mathcal{F}(B)\mathcal{F}(x)=\varphi(B)\mathcal{F}(x)$, i.e. $B*x=\varphi(B)x$. Hence $B*Bx=\varphi(B)x$. Note that $B$ is biprojection, then $B$ is a group-like projection \cite{LWW} i.e.
$$\Delta(B)(B\otimes 1)=\Delta(B)(1\otimes B)=B\otimes B.$$
Now we have
\begin{eqnarray*}
\varphi(B)x&=&B*(Bx)=(\varphi\otimes\iota)((B\otimes 1)\Delta(Bx))\\
&=&(\varphi\otimes \iota)((1\otimes B)\Delta(B)\Delta(x))\\
&=&\varphi(Bx)B,
\end{eqnarray*}
i.e. $x$ is a multiple of $B$.
\end{proof}

\begin{theorem}\label{hardy}[Hardy's uncertainty principle]
Suppose $\mathbb{G}$ is a unimodular Kac algebra and $w\in \mathbb{G}$ is a bi-shift of biprojection. For any $x\in L^1(\mathbb{G})\cap L^\infty(\mathbb{G})$, if $|x|\leq C|w|$ and $|\mathcal{F}(x)|\leq C'|\mathcal{F}(w)|,$ for some constants $C>0$ and $C'>0$,
then $x$ is a scalar multiple of $w$.
\end{theorem}

\begin{proof}
Suppose $w\in \mathbb{G}$ is a bi-shift of a biprojection $B$. Let $\widetilde{B}$ be the range projection of $\mathcal{F}(B)$, and $B_g$, $\widetilde{B}_h$ be right shifts of biprojections $B$, $\widetilde{B}$ respectively, such that $\mathcal{R}(w)\leq B_g$ and $\mathcal{R}(\mathcal{F}(w))\leq \widetilde{B}_h$.
If $x$ satisfies the assumption, then $\mathcal{R}(x)\leq B_g$ and $\mathcal{R}(\mathcal{F}(x))\leq \widetilde{B}_h$.
By Theorem \ref{Thm:main1}, we have that $\mathcal{R}(w)=\mathcal{R}(x)=B_g$ and $\mathcal{R}(\mathcal{F}(w))=\mathcal{R}(\mathcal{F}(x))=\widetilde{B}_h$.

We assume that $x\neq0$.
Then $xw^*$ and $ww^*$ are nonzero and
\begin{eqnarray*}
\mathcal{R}(\mathcal{F}(xw^*))&=&\mathcal{R}(\mathcal{F}(x)*\mathcal{F}(w^*))\\
&=&\mathcal{R}(\mathcal{F}(x)*\hat{R}(\mathcal{F}(w)^*)\\
&\leq &\mathcal{R}(\widetilde{B}_h*\hat{R}(\widetilde{B}_h)).
\end{eqnarray*}
By Theorem \ref{sq}, $\widetilde{B}_h*\hat{R}(\widetilde{B}_h)$ is a multiple of a projection and
$$\mathcal{S}(\mathcal{F}(xw^*))\leq \mathcal{S}(\widetilde{B}_h*\hat{R}(\widetilde{B}_h))=\mathcal{S}(\widetilde{B}_h)=\mathcal{S}(\mathcal{F}(w)).$$
Then
$$1\leq \mathcal{S}(xw^*)\mathcal{S}(\mathcal{F}(xw^*))=\mathcal{S}(wx^*)\mathcal{S}(\mathcal{F}(xw^*))\leq \mathcal{S}(w)\mathcal{S}(\mathcal{F}(w))=1.$$
Hence we have
$$\mathcal{S}(wx^*)=\mathcal{S}(w);\quad \mathcal{S}(\mathcal{F}(xw^*))=\mathcal{S}(\mathcal{F}(w))=\mathcal{S}(\widetilde{B}_h*\hat{R}(\widetilde{B}_h)).$$
Therefore
$$\mathcal{R}(wx^*)=\mathcal{R}(w)=\mathcal{R}(x)=\mathcal{R}(xw^*),\quad \mathcal{R}(\mathcal{F}(xw^*))=\mathcal{R}(\widetilde{B}_h*\hat{R}(\widetilde{B}_h)).$$
Hence $xw^*$ is a bi-shift of a biprojection. Similarly $ww^*$ is a bi-shift of a biprojection.
Moreover,
$$\mathcal{R}(wx^*)=\mathcal{R}(ww^*),\quad \mathcal{R}(\mathcal{F}(xw^*))=\mathcal{R}(\mathcal{F}(ww^*)).$$
By a similar argument, we have $(wx^*)*R(ww^*)^*$ and $(ww^*)*R(ww^*)^*$ are bi-shifts of biprojections and
\begin{equation}\label{ranq}
\begin{aligned}
\mathcal{R}((xw^*)*R(ww^*)^*)&=\mathcal{R}((ww^*)*R(ww^*)^*),\\
\mathcal{R}(\mathcal{F}((xw^*)*R(ww^*)^*))&=\mathcal{R}(\mathcal{F}((ww^*)*R(ww^*)^*)).
\end{aligned}
\end{equation}
By Theorem \ref{sq}, we have that $(ww^*)*R(ww^*)^*$ is a multiple of a biprojection $Q$.
By Lemma \ref{scale} and Equations (\ref{ranq}), we have that $(xw^*)*R(ww^*)^*$ is a multiple of biprojection $Q$.
Observe that both $x$ and $w$ are multiples of $(Q*(ww^*))w$.
Therefore $x$ is a scalar multiple of $w$.
\end{proof}

\begin{corollary}
Let $\mathbb{G}$ be a unimodular Kac algebra.
Suppose $B$ is a biprojection in $L^1(\mathbb{G})$ and $\widetilde{B}$ is the range projection of $\mathcal{F}(B)$ in $L^1(\widehat{\mathbb{G}})$.
Let $B_g$ and $\widetilde{B}_h$ be right shifts of biprojections $B$ and $\widetilde{B}$ respectively.
Then there is at most one element $x\in L^1(\mathbb{G})\cap L^2(\mathbb{G})$ up to a scalar such that the range projection of $x$ is contained in $B_g$ and the range projection of $\mathcal{F}(x)$ is contained in $\widetilde{B}_h$.
\end{corollary}

\begin{remark}
  Therefore we can use the supports $B_g$ and $\widetilde{B}_h$ to define a bi-shift of a biprojection. It is independent of the choice of $y$ in Definition \ref{bi-shift}.
\end{remark}

\end{document}